\documentclass[12pt,reqno]{amsart}
\usepackage{amsfonts}
\usepackage{amssymb,color}
\usepackage{amssymb}

\usepackage[colorlinks=true]{hyperref}
\hypersetup{urlcolor=blue, citecolor=blue}

\usepackage[margin=3cm, a4paper]{geometry}

\newtheorem{theo}{Theorem}[section]
\newtheorem{lemm}[theo]{Lemma}
\newtheorem{defi}[theo]{Definition}
\newtheorem{coro}[theo]{Corollary}

\numberwithin{equation}{section}

\newcommand{\bal}{\begin{align}}
\newcommand{\bbal}{\begin{align*}}
\newcommand{\beq}{\begin{equation}}
\newcommand{\eeq}{\end{equation}}
\newcommand{\bca}{\begin{cases}}
\newcommand{\eca}{\end{cases}}

\newcommand{\pa}{\partial}
\newcommand{\fr}{\frac}
\newcommand{\na}{\nabla}
\newcommand{\De}{\Delta}

\newcommand{\cd}{\cdot}
\newcommand{\ep}{\varepsilon}
\newcommand{\dd}{\mathrm{d}}

\newcommand{\R}{\mathbb{R}}

\newcommand{\D}{\mathrm{div}}
\newcommand{\bi}{\Big }
\newcommand{\Z}{\mathbb{Z}}

\begin{document}


\title[Higher Dimensional Camassa-Holm Equations]{Non-uniform dependence for higher dimensional Camassa-Holm equations in Besov spaces}

\author[J. Li]{Jinlu Li}
\address{School of Mathematics and Computer Sciences, Gannan Normal University, Ganzhou 341000, China}
\email{lijinlu@gnnu.edu.cn}

\author[w. Deng]{Wei Deng}
\address{School of Mathematics and Computer Sciences, Gannan Normal University, Ganzhou 341000, China}
\email{dwhh0406@163.com}

\author[M. Li]{Min Li}
\address{School of intormation Management , Jiangxi University Of Finance and Economics 330032, China}
\email{limin@jxufe.edu.cn}

\begin{abstract}
In this paper, we investigate the dependence on initial data of solutions to higher dimensional Camassa-Holm equations. We show that the data-to-solution map is not uniformly continuous dependence in Besov spaces $B^s_{p,r}(\R^d),s>\max\{1+\frac d2,\frac32\}$.
\end{abstract}

\maketitle

\section{Introduction and main result}

In the paper, we consider the following Cauchy problem of higher dimensional Camassa-Holm equations:
\begin{equation}\label{E-P}
\begin{cases}
\partial_tm+u\cdot \nabla m+\nabla u^Tm+(\mathrm{div} u)m=0, \qquad (t,x)\in \R^+\times \R^d,\\
m=(1-\De)u,\qquad (t,x)\in \R^+\times \R^d,\\
u(0,x)=u_0,\qquad x\in \R^d,
\end{cases}
\end{equation}
where $u=(u_1,u_2,\cdots,u_d)$ denotes the velocity of the fluid, $m=(m_1,m_2,\cdots,m_d)$ represents the momentum, or we can write \eqref{E-P} in components,
$$\partial_tm_i+\sum^d_{j=1}u_j\partial_{x_j}m_i+\sum^d_{j=1}(\partial_{x_i}u_j)m_j+m_i\sum^d_{j=1}\partial_{x_j}u_j=0,
~~~~~~~~i=1,2,\cdots,d.$$

While \eqref{E-P} is also called the Euler-Poincar\'{e} equations in the higher dimensional case $d\geq1$. In particular, when $d = 1,$ the system \eqref{E-P} is the classical Camassa-Holm (CH) equation, like the KdV equation, the CH equation  describes the unidirectional propagation of waves at the free surface of shallow water under the influence of gravity \cite{Camassa,Camassa.Hyman,Constantin.Lannes}.  It is completely integrable \cite{Camassa,Constantin-P}, has a bi-Hamiltonian structure \cite{Constantin-E,Fokas}, and admits exact peaked solitons of the form $ce^{-|x-ct|}$, $c>0$, which are orbitally stable \cite{Constantin.Strauss}. It is worth mentioning that the peaked solitons present the characteristic for the travelling water waves of greatest height and largest amplitude and arise as solutions to the free-boundary problem for incompressible Euler equations over a flat bed, cf. \cite{Constantin2,Constantin.Escher4,Constantin.Escher5,Toland}. The local well-posedness and ill-posedness for the Cauchy problem of the CH equation in Sobolev spaces and Besov spaces was discussed in \cite{Constantin.Escher,Constantin.Escher2,d1,d3,Guo-Yin,Li-Yin,Ro}. The non-uniform dependence on initial data for the CH equation was studied in \cite{H-K,H-K-M,Li-Yu-Zhu}. It was shown that there exist global strong solutions to the CH equation \cite{Constantin,Constantin.Escher,Constantin.Escher2} and finite time blow-up strong solutions to the CH equation \cite{Constantin,Constantin.Escher,Constantin.Escher2,Constantin.Escher3}. The existence and uniqueness of global weak solutions to the CH equation were proved in \cite{Constantin.Molinet, Xin.Z.P}. The global conservative and dissipative solutions of CH equation were discussed in \cite{Bressan.Constantin,Bressan.Constantin2}.

For higher dimensional Camassa-Holm equations (hd-CH), or the so-called Euler-Poincar\'{e} system \eqref{E-P} were first studied by Holm, Marsden, and Ratiu in 1998 as a framework for modeling and analyzing fluid dynamics \cite{H-M-R1,H-M-R2}, particularly for nonlinear shallow water waves, geophysical fluids and turbulence modeling. Later, hd-CH equations have many further interpretations beyond fluid applications. For instance, in 2-D, it is exactly the same as the averaged template matching equation for computer vision \cite{H-M-A}. Also, hd-CH equations have important applications in computational anatomy, it can be regarded as an evolutionary equation for a geodesic motion on a diffeomorphism group and it is associated with Euler-Poincar\'{e} reduction via symmetry (see, e.g, \cite{H-S-C,Y}).

 The rigorous analysis of higher dimensional Camassa-Holm equations \eqref{E-P} with $d\geq1$ was initiated by Chae and Liu \cite{Chae.Liu} who obtained the local well-posedness in Hilbert spaces $m_0\in H^{s+\frac d2},\ s\geq2$ and also gave a  blow-up criterion, zero $\alpha$ limit and the Liouville type theorem. Li, Yu and Zhai \cite{L.Y.Z} proved that the solution to \eqref{E-P} with a large class of smooth initial data blows up in finite time or exists globally in time, which reveals the nonlinear depletion mechanism hidden in the Euler-Poincar\'{e} system. By means of the Littlewood-Paley theory, Yan and Yin \cite{Y.Y} established the local existence and uniqueness in Besov spaces $B^s_{p,r},\ s>\max\{\frac32,1+\frac dp\}$ and $s=1+\frac dp,\ 1\leq p\leq 2d,\ r=1$. Lately, Li and Yin \cite{Li-Yin1} proved that the corresponding solution is continuous dependence for the initial data in Besov spaces. Inspired by \cite{H-H,H-M}, Li, Dai and Zhu \cite{L.D.Z} show that the corresponding solution is not uniformly continuous dependence for the initial data in Sobolev spaces $H^s(\R^d),s>1+\frac d2$.  For more results of higher dimensional Camassa-Holm equations, we refer the reads to see \cite{Luo-Yin,Z.Y.L}.

In this paper, motivated by \cite{Li-Yu-Zhu,H-K-M},  we will show that the solution map of \eqref{E-P} is not uniformly continuous depence in Besov space $B^s_{p,r}(\R^d),s>\max\{1+\frac d2,\frac32\}$. Himonas-Misiolek \cite{H-K-M} obtained the first result on the non-uniform dependence for the CH equation in $H^s(\mathbb{T})$ with $s\geq2$ using explicitly constructed travelling wave solutions, these types of construction solutions are only suitable in Sobolev spaces. Instead, in this paper we constructed a more general initial data to obtain the non-uniform dependence in Besov spaces.

According to \cite{Y.Y}, we can transform \eqref{E-P} into the following form:
\begin{align}\label{E-P1}
\partial_tu+u\cdot \nabla u= Q(u,u)+R(u,u),
\end{align}
where
\bbal
&Q(u,v)=-(1-\De)^{-1}\D\Big(\nabla u\nabla v+\nabla u(\nabla v)^T-(\nabla u)^T\nabla v
\\&\qquad \qquad -(\mathrm{div} u)\nabla v+\frac12\mathbf{I}(\nabla u:\nabla v)\Big),
\\&R(u,v)=-(1-\De)^{-1}\Big((\mathrm{div} u)v+ u\cd \na v\Big).
\end{align*}

Then, we have the following result.

\begin{theo}\label{th2}
Let $d\geq 2, 1\leq p,r\leq \infty$ and $s>\max\{1+\frac dp,\frac32\}$. The data-to-solution map for higher dimensional Camassa-Holm equations \eqref{E-P} is not uniformly continuous from any bounded subset in $B^s_{p,r}$ into $\mathcal{C}([0,T];B^s_{p,r}(\R^d))$. That is, there exists two sequences of solutions $u^n$ and $v^n$ such that
\bbal
&||u^n_0||_{B^s_{p,r}(\R^d)}+||v^n_0||_{B^s_{p,r}(\R^d)}\lesssim 1, \quad \lim_{n\rightarrow \infty}||u^n_0-v^n_0||_{B^s_{p,r}(\R^d)}= 0,
\\&\liminf_{n\rightarrow \infty}||u^n(t)-v^n(t)||_{B^s_{p,r}(\R^d)}\gtrsim t,  \quad t\in[0,T_0],
\end{align*}
with small time $T_0$ for $T_0\leq T$.
\end{theo}

Our paper is organized as follows. In Section 2, we give some preliminaries which will be used in the sequel. In Section 3, we give the proof of our main theorem.\\

\noindent\textbf{Notations.} Given a Banach space $X$, we denote its norm by $\|\cdot\|_{X}$. The symbol $A\lesssim B$ means that there is a uniform positive constant $c$ independent of $A$ and $B$ such that $A\leq cB$. Here
\bbal
&(\na u^T)_{i,j}=\pa_{x_i}u^j, \quad (u\cd \na v)_i=\sum^d_{k=1}u_k\pa_{x_k}u_i, \quad (\na u\na v)_{ij}=\sum^d_{k=1}\pa_{x_i}u_k\pa_{x_k}v_j,
\\&\na u:\na v=\sum^d_{i,j=1}\pa_{x_i}u_j\pa_{x_i}v_j.
\end{align*}

\section{Littlewood-Paley analysis}

In this section, we will recall some facts about the Littlewood-Paley decomposition, the nonhomogeneous Besov spaces and their some useful properties. For more details, the readers can refer to \cite{B.C.D}.

There exists a couple of smooth functions $(\chi,\varphi)$ valued in $[0,1]$, such that $\chi$ is supported in the ball $\mathcal{B}\triangleq \{\xi\in\mathbb{R}^d:|\xi|\leq \frac 4 3\}$, and $\varphi$ is supported in the ring $\mathcal{C}\triangleq \{\xi\in\mathbb{R}^d:\frac 3 4\leq|\xi|\leq \frac 8 3\}$. Moreover,
$$\forall\,\ \xi\in\mathbb{R}^d,\,\ \chi(\xi)+{\sum\limits_{j\geq0}\varphi(2^{-j}\xi)}=1,$$
$$\forall\,\ 0\neq\xi\in\mathbb{R}^d,\,\ {\sum\limits_{j\in \Z}\varphi(2^{-j}\xi)}=1,$$
$$|j-j'|\geq 2\Rightarrow\textrm{Supp}\,\ \varphi(2^{-j}\cdot)\cap \textrm{Supp}\,\ \varphi(2^{-j'}\cdot)=\emptyset,$$
$$j\geq 1\Rightarrow\textrm{Supp}\,\ \chi(\cdot)\cap \textrm{Supp}\,\ \varphi(2^{-j}\cdot)=\emptyset.$$
Then, we can define the nonhomogeneous dyadic blocks $\Delta_j$ and nonhomogeneous low frequency cut-off operator $S_j$ as follows:
$$\Delta_j{u}= 0,\,\ if\,\ j\leq -2,\quad
\Delta_{-1}{u}= \chi(D)u=\mathcal{F}^{-1}(\chi \mathcal{F}u),$$
$$\Delta_j{u}= \varphi(2^{-j}D)u=\mathcal{F}^{-1}(\varphi(2^{-j}\cdot)\mathcal{F}u),\,\ if \,\ j\geq 0,$$
$$S_j{u}= {\sum\limits_{j'=-\infty}^{j-1}}\Delta_{j'}{u}.$$

\begin{defi}[\cite{B.C.D}]\label{de2.3}
Let $s\in\mathbb{R}$ and $1\leq p,r\leq\infty$. The nonhomogeneous Besov space $B^s_{p,r}$ consists of all tempered distribution $u$ such that
\begin{align*}
||u||_{B^s_{p,r}(\R^d)}\triangleq \Big|\Big|(2^{js}||\Delta_j{u}||_{L^p(\R^d)})_{j\in \Z}\Big|\Big|_{\ell^r(\Z)}<\infty.
\end{align*}
\end{defi}

Then, we have the following product laws.
\begin{lemm}[\cite{B.C.D}]\label{le-pro}
(1) For any $s>0$ and $1\leq p,r\leq\infty$, there exists a positive constant $C=C(d,s,p,r)$ such that
$$\|uv\|_{B^s_{p,r}(\mathbb{R}^d)}\leq C\Big(\|u\|_{L^{\infty}(\mathbb{R}^d)}\|v\|_{B^s_{p,r}(\mathbb{R}^d)}+\|v\|_{L^{\infty}(\mathbb{R}^d)}\|u\|_{B^s_{p,r}(\mathbb{R}^d)}\Big).$$
(2) Let $1\leq p,r\leq\infty$ and $s>\max\{\frac32,1+\frac{d}{p}\}$. Then, we have
$$||uv||_{B^{s-2}_{p,r}(\mathbb{R}^d)}\leq C||u||_{B^{s-1}_{p,r}(\mathbb{R}^d)}||v||_{B^{s-2}_{p,r}(\mathbb{R}^d)}.$$
\end{lemm}

\begin{lemm}[Theorem 3.38, \cite{B.C.D} and Lemma 2.9, \cite{Li-Yin1}]\label{lem:TDe}
Let $1\leq p,r\leq \infty$. Assume that
\begin{align}
\sigma> -d \min(\frac{1}{p}, \frac{1}{p'}) \quad \mathrm{or}\quad \sigma> -1-d \min(\frac{1}{p}, \frac{1}{p'})\quad \mathrm{if} \quad \mathrm{div\,} v=0.
\end{align}

Assume that $f_0\in B^\sigma_{p,r}$, $g\in L^1(0,T; B^\sigma_{p,r})$, and $\nabla v$ belongs to $L^1(0,T; B^{\sigma-1}_{p,r})$ if $\sigma> 1+{\frac{d}{p}}$ (or
$\sigma =1+\frac{d}{p}\ \ and \ \ r=1$) or to $L^1(0,T; B^{\frac{d}{p}}_{p,r}\bigcap L^\infty)$ otherwise. If $f\in L^\infty(0,T; B^\sigma_{p,r})\bigcap \mathcal{C}([0,T]; \mathcal{S}^{'})$ solves the following linear transport equation:

\[(T)\left\{\begin{array}{l}
\partial_t f+v\cdot\nabla f=g,\\
f|_{t=0} =f_0,
\end{array}\right.\]

then there exists a constant $C=C(d,p,r,\sigma)$ such that the following statements hold:
\begin{align}\label{ES2}
\sup_{s\in [0,t]}\|f(s)\|_{B^{\sigma}_{p,r}}\leq Ce^{CV_{p}(v,t)}\Big(\|f_0\|_{B^\sigma_{p,r}}
+\int^t_0\|g(\tau)\|_{B^{s}_{p,r}}\dd \tau\Big),
\end{align}
with
\begin{align*}
V_{p}(v,t)=
\begin{cases}
\int_0^t \|\nabla v(s)\|_{B^{\frac{d}{p}}_{p,\infty}\cap L^\infty}\dd s,&\quad\mathrm{if} \; \sigma<1+\frac{d}{p},\\
\int_0^t \|\nabla v(s)\|_{B^{\sigma}_{p,r}}\dd s,&\quad\mathrm{if} \; \sigma=1+\frac{d}{p} \mbox{ and } r>1,\\
\int_0^t \|\nabla v(s)\|_{B^{\sigma-1}_{p,r}}\dd s, &\quad \mathrm{if} \;\sigma>1+\frac{d}{p}\ \mathrm{or}\ \{\sigma=1+\frac{d}{p} \mbox{ and } r=1\}.
\end{cases}
\end{align*}
If $f=v$, then for all $\sigma>0$ ($\sigma>-1$, if $\mathrm{div\,} v=0$), the estimate \eqref{ES2} holds with
\[V_{p}(t)=\int_0^t \|\nabla v(s)\|_{L^\infty}\dd s.\]
\end{lemm}

\section{Non-uniform continuous dependence}

In this section, motivated by \cite{Li-Yu-Zhu}, we will give the proof of our main theorem. First, we can construct a sequence initial data $u^n_0=f_n$, which can approximate to the solution $\mathbf{S}_t(u^n_0)$. Lately, we can construct a sequence initial data $v^n_0=f_n+g_n$ and can approximate to the solution $\mathbf{S}_t(v^n_0)$ by $v^n_0+tg_n\cdot\nabla f_n$. Finally, by the precious steps, we can conclude that the solution map is not uniformly continuous. In order to state our main result, we first recall the following local-in-time existence of strong solutions to \eqref{E-P} in \cite{Y.Y}:
\begin{lemm}[\cite{Y.Y}]\label{le10}
For $1\leq p,r\leq \infty$ and $s>\max\{1+\frac dp,\frac32\}$ and initial data $u_0\in B^s_{p,r}(\R^d)$, there exists a time $T=T(s,p,r,d,||u_0||_{B^s_{p,r}(\R^d)})>0$ such that the system \eqref{E-P} have a unique solution $u\in \mathcal{C}([0,T];B^s_{p,r}(\R^d))$. Moreover, for all $t\in[0,T]$, there holds
\[||u(t)||_{B^s_{p,r}(\R^d)}\leq C||u_0||_{B^s_{p,r}(\R^d)}.\]
\end{lemm}

\begin{coro}\label{co1}
Let $1\leq p,r\leq \infty$ and $s>\max\{1+\frac dp,\frac32\}$. Assume that $u\in \mathcal{C}([0,T];B^s_{p,r})$ be the solution of the system \eqref{E-P}. Then, we have for all $t\in[0,T]$,
\bbal
||u(t)||_{B^{s-1}_{p,r}(\R^d)}\leq ||u_0||_{B^{s-1}_{p,r}(\R^d)}e^{C\int^t_0||u(\tau)||_{B^s_{p,r}(\R^d)}\dd \tau},
\end{align*}
and
\bbal
||u(t)||_{B^{s+1}_{p,r}(\R^d)}\leq ||u_0||_{B^{s+1}_{p,r}(\R^d)}e^{C\int^t_0||u(\tau)||_{B^s_{p,r}(\R^d)}\dd \tau}.
\end{align*}
\end{coro}
\begin{proof}
The results can easily deduce from Lemma \ref{lem:TDe} and Gronwall's inequality. Here, we omit it.
\end{proof}

Now, we give the details of the proof to our theorem. \\
\textbf{Proof of the main theorem.} Letting $\hat{\phi}$ be a $C_0(\mathbb{R})$ such that
\begin{equation*}
\hat{\phi}(x)=
\begin{cases}
1, \quad |x|\leq \frac{1}{4^d},\\
0, \quad |x|\geq \frac{1}{2^d}.
\end{cases}
\end{equation*}
First, we choose the velocity $u^n_0$ having the following form:
\bbal
u^n_0=
\Big(f_n,0,\cdots,0\Big),
\end{align*}
with
\bbal
f_n(x)=
2^{-ns}\phi(x_1)\sin(\frac{17}{12}x_1)\phi(x_2)\cdots \phi(x_d),
\qquad n \in \Z.
\end{align*}
An easy computation gives that
\bbal
\hat{f}_n=2^{-ns-1}i\bi[\hat{\phi}
\bi(\xi_1+\frac{17}{12}2^n\bi)-\hat{\phi}\bi(\xi_1-\frac{17}{12}2^n\bi)\bi]
\hat{\phi}(\xi_2)\cdots \hat{\phi}(\xi_d),
\end{align*}
which implies
\bbal
\mathrm{supp} \ \hat{f}_n\subset \Big\{\xi\in\R^d: \ \frac{17}{12}2^n-\fr12\leq |\xi|\leq \frac{17}{12}2^n+\fr12\Big\}.
\end{align*}
Then, we deduce that
\bbal
\Delta_j(f_n)=
\begin{cases}
f_n, \quad  j=n,\\
0, \qquad j\neq n.
\end{cases}
\end{align*}
On account of Definition \ref{de2.3}, we can show that for $k\in \R$,
\bbal
||u^n_0||_{B^{s+k}_{p,r}}\leq C2^{kn}.
\end{align*}
Let $u^n$ be the solution of \eqref{E-P1} with initial data $u^n_0$. Then, we have the following estimate between $u^n_0$ and $u^n$.
\begin{lemm}\label{lemm5}
Let $\ep_s=\frac12\min\{s-1-\frac dp,s-\frac32,1\}$, then there holds
$$||u^n-u^n_0||_{L^\infty_T(B^s_{p,r})}\leq C2^{-n\ep_s},$$
where $T\simeq 1$.
\end{lemm}

\begin{proof}
By the well-posedness result (see Lemma \ref{le10}), the solution $u^n$ belong to $\mathcal{C}([0,T];B^s_{p,r})$ and have lifespan $T\simeq 1$. In fact, it is easy to show that for $k>-s$,
\bbal
||u^n||_{L^\infty_T(B^{s+k}_{p,r})}\leq C2^{kn}.
\end{align*}
Since $s-\ep_s-1>\frac dp$, then we have $||f||_{L^\infty}\leq C||f||_{B^{s-1-\ep_s}_{p,r}}$. Hence, we obtain
\bbal
||u^n-u^n_0||_{B^s_{p,r}}&\leq \int^t_0||\pa_\tau u^n||_{B^s_{p,r}} \dd\tau
\\&\leq ||u^n||_{L^\infty}||u^n||_{B^{s+1}_{p,r}}
+||\na u^n||_{L^\infty}||u^n||_{B^{s}_{p,r}}
\\&\leq ||u^n||_{B^{s-1-\ep_s}_{p,r}}||u^n||_{B^{s+1}_{p,r}}
+||\na u^n||_{B^{s-1-\ep_s}_{p,r}}||u^n||_{B^{s}_{p,r}}
\\&\leq 2^{-n\ep_s}.
\end{align*}
\end{proof}

Now, we choose the velocity $v^n_0$ having the following form:
\bbal
v^n_0=
\Big(f_n+g_n,0,\cdots,0\Big),
\end{align*}
with
\bbal
g_n(x)=
2^{-n}\phi(x_1)\phi(x_2)\cdots \phi(x_d),
\qquad n \in \Z.
\end{align*}
Direct calculation shows that for $k\geq -1$,
\bal\label{v0-es}
||v^n_0||_{B^{s+k}_{p,r}}\leq C2^{kn}, \qquad ||v^n_0,\na v^n_0||_{L^\infty}\leq C(2^{-n}+2^{-(s-1)n}).
\end{align}
Let $v^n$ be the solution of \eqref{E-P1} with initial data $v^n_0$. Then, we have the following estimate between $v^n_0$ and $v^n$.
\begin{lemm}\label{lemm6}
Let $\ep_s=\frac12\min\{s-1-\frac dp,s-\frac32,1\}$, then there holds for $t\leq T$
$$||v^n(t,\cdot)-v^n_0+tv^n_0\cd\na v^n_0||_{B^s_{p,r}}\leq C2^{-n\ep_s}+Ct^2.$$
\end{lemm}
\begin{proof}
By the well-posedness result (see Lemma \ref{le10}), the solution $v^n$ belong to $\mathcal{C}([0,T];B^s_{p,r})$ and have common lifespan $T\simeq 1$. It is easy to check that
\bal\label{v-(s-1)}
||v^n_0,V^n_0,v^n||_{B^{s+k}_{p,r}}\leq C2^{kn}, \quad \mathrm{for} \quad k=-1,0,1.
\end{align}
For simplicity, we denote $w_n=v^n-v^n_0-tV^n_0$ with $V^n_0=-v^n_0\cd\na v^n_0$, then we can deduce that
\bal\label{wn(1)}
\pa_tw_n+v^n\cd\na w_n=&-t(V^n_0\cd\na v^n_0+v^n_0\cd\na V^n_0)-t^2V^n_0\cd\na V^n_0\nonumber
\\&-w_n\cd\na(v^n_0+tV^n_0)+ Q(v^n,v^n)+R(v^n,v^n).
\end{align}
From Lemma \ref{le-pro}, we have
\bal\label{R(v,v)-(s-1)}
||R(v^n,v^n)||_{B^{s-1}_{p,r}}&\leq C||v_n||^2_{B^{s-1}_{p,r}}\leq C2^{-2n},
\end{align}
and
\bal\label{R(v,v)-(s)}
||R(v^n,v^n)||_{B^{s}_{p,r}}&\leq C||\na v_n||_{B^{s-1}_{p,r}}||v_n||_{B^{s-1}_{p,r}}
\leq C2^{-n}.
\end{align}
For the term $Q(v^n,v^n)$, we need decompose it as follows:
\bal\label{Q-de}
Q(v^n,v^n)=Q(w^n,v^n)+Q(v^n_0,w^n)+Q(v^n_0,tV^n_0)+Q(tV^n_0,v^n)+Q(v^n_0,v^n_0).
\end{align}
Then, combining \eqref{wn(1)} and \eqref{Q-de}, we see that
\bbal
\pa_tw_n+v^n\cd\na w_n=&\underbrace{-t(V^n_0\cd\na v^n_0+v^n_0\cd\na V^n_0+tV^n_0\cd\na V^n_0-Q(V^n_0,v^n)-Q(v^n_0,V^n_0))}_{I}
\\&\underbrace{-w_n\cd\na(v^n_0+tV^n_0)+Q(w^n,v^n)+Q(v^n_0,w^n)}_{II}
\\&+ \underbrace{Q(v^n_0,v^n_0)+R(v^n,v^n)}_{III}.
\end{align*}
By Lemma \ref{le-pro} and \eqref{v-(s-1)}, we obtain
\bal\label{I-(s-1)}
||I(t)||_{B^{s-1}_{p,r}}\leq C||v^n,v^n_0,V^n_0||_{B^{s-1}_{p,r}}||v^n,v^n_0,V^n_0||_{B^s_{p,r}} \leq  Ct2^{-n},
\end{align}
and
\bal\label{II-(s-1)}
||II(t)||_{B^{s-1}_{p,r}}\leq C||w_n||_{B^{s-1}_{p,r}}||v^n,v^n_0,V^n_0||_{B^s_{p,r}}\leq C||w_n||_{B^{s-1}_{p,r}}.
\end{align}
For the term $III(t)$, we just estimate $Q(v^n_0,v^n_0)$. It follows from Lemma \ref{le-pro} and \eqref{v0-es} that
\bbal
||Q(v^n_0,v^n_0)||_{B^{s-1}_{p,r}}&\leq C||Q(v^n_0,v^n_0)||_{B^{s-\frac12}_{p,r}}\leq C||v^n_0||_{B^{s-\frac12}_{p,r}}||v^n_0,\na v^n_0||_{L^\infty}
\\&\leq C2^{-\frac12n}\big(2^{-n(s-1)}+2^{-n}\big)\leq C2^{-n(1+\ep_s)},
\end{align*}
which along with \eqref{R(v,v)-(s-1)} implies
\bal\label{III-(s-1)}
||III(t)||_{B^{s-1}_{p,r}}\leq C2^{-n(1+\ep_s)}.
\end{align}
Thus, we infer from Lemma \ref{lem:TDe} and \eqref{I-(s-1)}-\eqref{III-(s-1)} that
\bbal
||w_n||_{B^{s-1}_{p,r}}\leq \int^t_0||w_n||_{B^{s-1}_{p,r}}\dd \tau+Ct^2+C2^{-n(1+\ep_s)},
\end{align*}
which yields
\bal\label{wn-es}
||w_n||_{B^{s-1}_{p,r}}\leq C2^{-n}t^2+C2^{-n(1+\ep_s)}.
\end{align}
Now, we can tackle with the $B^{s}_{p,r}$ norm of $w_n$. According to Lemma \ref{le-pro} and \eqref{v-(s-1)}, we have
\bal\label{I-(s)}
||I(t)||_{B^s_{p,r}}\leq C||v^n,v^n_0,V^n_0||_{B^{s-1}_{p,r}}||v^n,v^n_0,V^n_0||_{B^{s+1}_{p,r}}
+C||v^n,v^n_0,V^n_0||^2_{B^s_{p,r}}
 \leq Ct,
\end{align}
and
\bal\label{II-(s)}
||II(t)||_{B^s_{p,r}}&\leq C||w_n||_{B^{s}_{p,r}}||v^n,v^n_0,V^n_0||_{B^s_{p,r}}
+C||w_n||_{B^{s+1}_{p,r}}||v^n,v^n_0,V^n_0||_{B^{s-1}_{p,r}} \nonumber
\\&\leq C||w_n||_{B^{s}_{p,r}}+C2^{n}||w_n||_{B^{s-1}_{p,r}}.
\end{align}
From Lemma \ref{le-pro} and \eqref{v0-es}, we obtain
\bbal
||Q(v^n_0,v^n_0)||_{B^{s}_{p,r}}&\leq C||v^n_0||_{B^{s}_{p,r}}||v^n_0,\na v^n_0||_{L^\infty}
\\&\leq C||v^n_0||_{B^{s}_{p,r}}||v^n_0,\na v^n_0||_{B^{s-1-\ep_s}_{p,r}}\leq C2^{-n\ep_s},
\end{align*}
which along with \eqref{R(v,v)-(s)} leads to
\bal\label{III-(s)}
||III(t)||_{B^s_{p,r}}\leq C2^{-n\ep_s}.
\end{align}
According to Lemma \ref{lem:TDe} and \eqref{wn-es}, \eqref{I-(s)}-\eqref{III-(s)}, we have
\bbal
||w_n||_{B^{s}_{p,r}}&\leq C\int^t_0||w_n||_{B^{s}_{p,r}}\dd \tau+\int^t_0||w_n||_{B^{s-1}_{p,r}}||v^n_0,V^n_0||_{B^{s+1}_{p,r}}\dd \tau+Ct^2+C2^{-n\ep_s}
\\&\leq C\int^t_0||w_n||_{B^{s}_{p,r}}\dd \tau+Ct^2+C2^{-n\ep_s},
\end{align*}
which gives
\bbal
||w_n||_{B^{s}_{p,r}}\leq Ct^2+C2^{-n\ep_s}.
\end{align*}
\end{proof}

Finally, we prove the result of Theorem \ref{th2}. It is easy to show that
\begin{align}\label{ESend}
||u^n_0-v^n_0||_{B^s_{p,r}}\leq ||g_n||_{B^s_{p,r}}\leq C2^{-n},
\end{align}
which tend to 0 for $n$ tends to infinity. For the solution, using of triangle inequality,
\bbal
&\quad \ ||u^n-v^n||_{B^s_{p,r}}
\\&=||(u^n-u^n_0)-(v^n-v^n_0+tv^n_0\cd\na v^n_0)+u^n_0-v^n_0+tv^n_0\cd\na v^n_0||_{B^s_{p,r}}
\\&\geq ||tv^n_0\cd\na v^n_0||_{B^s_{p,r}}-||u^n_0-v^n_0||_{B^s_{p,r}}-||u^n-u^n_0||_{B^s_{p,r}}-||v^n-v^n_0+tv^n_0\cd\na v^n_0||_{B^s_{p,r}}.
\end{align*}
Thanks for  Lemma \ref{lemm5} and Lemma \ref{lemm6}, we have
\begin{align}\label{393}
||u^n-v^n||_{B^s_{p,r}}
\geq ||tv^n_0\cd\na v^n_0||_{B^s_{p,r}}-Ct^2-C2^{-n\ep_s}.
\end{align}
As $\big(v^n_0\cd\na v^n_0\big)_i=0$ for $i=2,\cdots d$, and
\begin{align}\label{397}
\big(v^n_0\cd\na v^n_0\big)_1=f_n\cd\na f_n+f_n\cd\na g_n+g_n\cd\na f_n+g_n\cd\na g_n.
\end{align}
By the definition of $f_n$ and $g_n,$ we have
\bbal
&\quad \ ||f_n\cd\na f_n,g_n\cd\na g_n,f_n\cd\na g_n||_{B^s_{p,r}}
\\&\leq C||f_n,g_n||_{L^\infty}||f_n,g_n||_{B^s_{p,r}}+C||\na f_n,g_n||_{L^\infty}||f_n,g_n||_{B^s_{p,r}}
\\&\leq C2^{-n\ep_s},
\end{align*}
which along with \eqref{393} and \eqref{397} implies
\bal\label{1000}
||u^n-v^n||_{B^s_{p,r}}\geq ct||g_n\cd\na f_n||_{B^s_{p,r}}-Ct^2-C2^{-n\ep_s}.
\end{align}
Noticing that
\bal\label{1001}
||g_n\cd\na f_n||_{B^s_{p,r}}&=2^{ns}||g_n\cd\na f_n||_{L^p}\nonumber
\\&\geq||\phi^2(x_1)\cos(\frac{17}{12}2^nx_1)||_{L^p}||\phi||^{2(d-1)}_{L^p}-C2^{-n\ep_s}.
\end{align}
Substituting \eqref{1001} into \eqref{1000}, we obtain
\bbal
||u^n-v^n||_{B^s_{p,r}}\geq ct-Ct^2-C2^{-n\ep_s}.
\end{align*}
Restricting $T_0$ to a small range,  by choosing another $c$ we have
\begin{align}\label{ESlast}
||u^n(t)-v^n(t)||_{B^s_{p,r}(\R^d)}\geq ct,  \quad t\in[0,T_0].
\end{align}
Letting $n$ go to $\infty$, then \eqref{ESlast} together with \eqref{ESend} complete the proof of Theorem \ref{th2}.

\vspace*{1em}
\noindent\textbf{Acknowledgements.}  J. Li is supported by the National Natural Science Foundation of China (Grant No.11801090).

\end{document}